\documentclass[reqno,a4paper]{amsart}

\usepackage[T1]{fontenc}
\usepackage[utf8]{inputenc}
\usepackage{lmodern}
\usepackage{microtype}

\usepackage{fullpage}

\usepackage{amssymb}

\usepackage[]{hyperref} 

\usepackage[nobysame,alphabetic,initials,msc-links]{amsrefs}

\DefineSimpleKey{bib}{how}
\renewcommand{\PrintDOI}[1]{%
  \href{http://dx.doi.org/#1}{{\tt DOI:#1}}%
}
\renewcommand{\eprint}[1]{#1}
\BibSpec{book}{%
    +{}  {\PrintPrimary}                {transition}
    +{.} { \PrintDate}                  {date}
    +{.} { \textit}                     {title}
    +{.} { }                            {part}
    +{:} { \textit}                     {subtitle}
    +{,} { \PrintEdition}               {edition}
    +{}  { \PrintEditorsB}              {editor}
    +{,} { \PrintTranslatorsC}          {translator}
    +{,} { \PrintContributions}         {contribution}
    +{,} { }                            {series}
    +{,} { \voltext}                    {volume}
    +{,} { }                            {publisher}
    +{,} { }                            {organization}
    +{,} { }                            {address}
    +{,} { }                            {status}
    +{,} { \PrintDOI}                   {doi}
    +{,} { \PrintISBNs}                 {isbn}
    +{}  { \parenthesize}               {language}
    +{}  { \PrintTranslation}           {translation}
    +{;} { \PrintReprint}               {reprint}
    +{.} { }                            {note}
    +{.} {}                             {transition}
    +{}  {\SentenceSpace \PrintReviews} {review}
}
\BibSpec{article}{%
    +{}  {\PrintAuthors}                {author}
    +{,} { \textit}                     {title}
    +{.} { }                            {part}
    +{:} { \textit}                     {subtitle}
    +{,} { \PrintContributions}         {contribution}
    +{.} { \PrintPartials}              {partial}
    +{,} { }                            {journal}
    +{}  { \textbf}                     {volume}
    +{}  { \PrintDatePV}                {date}
    +{,} { \issuetext}                  {number}
    +{,} { \eprintpages}                {pages}
    +{,} { }                            {status}
    +{,} { \PrintDOI}                   {doi}
    +{,} { \eprint}        {eprint}
    +{}  { \parenthesize}               {language}
    +{}  { \PrintTranslation}           {translation}
    +{;} { \PrintReprint}               {reprint}
    +{.} { }                            {note}
    +{.} {}                             {transition}
    +{}  {\SentenceSpace \PrintReviews} {review}
}
\BibSpec{collection.article}{%
    +{}  {\PrintAuthors}                {author}
    +{,} { \textit}                     {title}
    +{.} { }                            {part}
    +{:} { \textit}                     {subtitle}
    +{,} { \PrintContributions}         {contribution}
    +{,} { \PrintConference}            {conference}
    +{}  {\PrintBook}                   {book}
    +{,} { }                            {booktitle}
    +{,} { \PrintDateB}                 {date}
    +{,} { pp.~}                        {pages}
    +{,} { }                            {publisher}
    +{,} { }                            {organization}
    +{,} { }                            {address}
    +{,} { }                            {status}
    +{,} { \PrintDOI}                   {doi}
    +{,} { \eprint}        {eprint}
    +{}  { \parenthesize}               {language}
    +{}  { \PrintTranslation}           {translation}
    +{;} { \PrintReprint}               {reprint}
    +{.} { }                            {note}
    +{.} {}                             {transition}
    +{}  {\SentenceSpace \PrintReviews} {review}
}
\BibSpec{misc}{%
  +{}{\PrintAuthors}  {author}
  +{,}{ \textit}      {title}
  +{.}{ }             {how}
  +{}{ \parenthesize} {date}
  +{,} { available at \eprint}        {eprint}
  +{,}{ available at \url}{url}
  +{,}{ }             {note}
  +{.}{}              {transition}
}

\usepackage{tikz}
\usetikzlibrary{knots,cd,backgrounds}
\usetikzlibrary{math}

\numberwithin{equation}{section}

\newtheorem{theorem}{Theorem}[section]

\newtheorem{proposition}[theorem]{Proposition}

\theoremstyle{remark}
\newtheorem{remark}[theorem]{Remark}

\theoremstyle{definition}

\newcommand{\bp}{\begin{proof}}
\newcommand{\ep}{\end{proof}}

\mathchardef\mhyph="2D 				

\newcommand{\R}{\mathbb{R}}

\newcommand{\bC}{\mathbb{C}}

\newcommand{\fA}{\mathfrak{A}}
\newcommand{\fB}{\mathfrak{B}}
\newcommand{\fC}{\mathfrak{C}}
\newcommand{\fM}{\mathfrak{M}}
\newcommand{\fN}{\mathfrak{N}}

\DeclareMathOperator{\Tr}{Tr}

\newcommand{\sa}{\mathrm{sa}}
\newcommand{\amn}{\mathbin{\triangledown}}

\newcommand{\hmn}{\mathbin{!}}

\begin{document}

\title{Lieb type convexity for positive operator monotone decreasing functions}

\author{Hans Henrich Neumann}

\email{hanshne@math.uio.no}

\author{Makoto Yamashita}

\email{makotoy@math.uio.no}

\address{Department of Mathematics, University of Oslo, P.O box 1053, Blindern, 0316 OSLO, Norway}

\thanks{M.Y.~is supported by the NFR funded project 300837 ``Quantum Symmetry'' and JSPS Kakenhi 18K13421.}


\date{v2: 17.06.2021, minor modifications; v1: 02.06.2021}

\begin{abstract}
We prove Lieb type convexity and concavity results for trace functionals associated with positive operator monotone (decreasing) functions and certain monotone concave functions.
This gives a partial generalization of Hiai's recent work on trace functionals associated with power functions.
\end{abstract}

\maketitle

\section{Introduction}

In a breakthrough paper on concavity of quantum entropy \cite{MR0332080}, Lieb proved concavity and convexity properties of functionals of the form
\[
(A, B) \mapsto \Tr(A^p K B^q K^*)
\]
defined on the space of positive definite matrices.
Since then this has been expanded in many ways, see \citelist{\cite{MR0343073}\cite{MR1730503}\cite{MR2379699}\cite{MR3067823}\cite{MR3429039}\cite{MR4064777}} and references therein.
One definitive form is given by Hiai \cite{MR3464069}, who proved (among other configurations) the joint convexity of functionals of the form
\[
(A, B) \mapsto \Tr h \bigl(\Phi(A^{-p})^{-1/2} \Psi(B^{-q}) \Phi(A^{-p})^{-1/2}\bigr) \quad (A, B > 0)
\]
for strictly positive maps $\Phi\colon M_m(\bC) \to M_k(\bC)$, $\Psi\colon M_n(\bC) \to M_k(\bC)$, $0 < p, q \le 1$, and certain nondecreasing concave function $h$.
His proof is based on an elegant use of the variational formula for trace functionals based on the Legendre transform which can be traced back to \cite{MR2379699}, together with intricate relation between operator majorization and matrix norms.

In this short note we prove a variant of this result, where we allow functional calculus by arbitrary positive operator monotone decreasing functions inside positive maps instead of power maps, but impose a stronger condition on $h$.
A similar result for the geometric mean of positive matrices \cite{MR0420302} was recently proved by Kian and Seo \cite{MR4227812}.

Our proof is a combination the variational method with concave functions, and one-variate convexity of operator valued maps of the form $h(\Phi(f(A)))$ for operator monotone $h$ and positive operator monotone decreasing $f$ established by Kirihata and the second named author in \cite{MR4156435}, which was motivated by certain operator log-convexity of such $f$ due to Ando and Hiai \cite{MR2805638}.

Besides allowing a bigger class of functions inside the positive maps $\Phi$ and $\Psi$, we also give analogous results in the framework of C$^*$-algebras endowed with tracial states.
The overall strategy is essentially the same as the case of matrices, but we rely on Petz's work \citelist{\cite{MR796042}\cite{MR870784}} on trace inequalities for tracial von Neumann algebras, and an adaptation of Hiai's variational formula to this setting.

\section{Preliminaries}

We denote the set of positive invertible matrices by $M_n(\bC)^{++}$, and the set of selfadjoint matrices by $M_n(\bC)_{\sa}$.
A linear map $\Phi\colon M_n(\bC) \to M_k(\bC)$ is \emph{strictly positive} if it sends $M_n(\bC)^{++}$ into $M_k(\bC)^{++}$.
A real function $f(x)$ for $x > 0$ is \emph{operator monotone} when the functional calculus $f(A)$ for $A \in M_n(\bC)^{++}$ with arbitrary $n$ preserves order relation, that is, $A \le B$ in $M_n(\bC)^{++}$ implies $f(A) \le f(B)$.
For details, we refer to standard references such as \cite{MR1477662}.

For positive numbers $a$ and $b$, we denote their arithmetic and harmonic means by
\begin{align*}
a \amn b &= \frac{a + b}2,&
a \hmn b &= \frac2{a^{-1} + b^{-1}}.
\end{align*}
These admit obvious generalization to operator transforms for $A, B \in M_n(\bC)^{++}$.

Let $h(x)$ be a nondecreasing and concave function for $x > 0$, satisfying $\lim_{x \to \infty} h(x) x^{-1} = 0$.
Its \emph{Legendre transform} is given by
\[
\check h(t) = \inf_{x > 0} t x - h(x).
\]
Then $\check h$ satisfies the same assumptions as $h$ \cite{MR3464069}*{Lemma A.2}.

\medskip
Now, let us list key ingredients of our proof.
First is the variational formula for trace functionals associated with concave functions.

\begin{proposition}[\cite{MR3464069}*{Lemma A.2}]\label{prop:Hia-Tr-h-A-by-Leg-tra}
Let $h$ be as above.
For any positive matrix $A \in M_n(\bC)^{++}$, we have
\[
\Tr h(A) = \inf_{B \in M_n(\bC)^{++}} \Tr(A B - \check h(B)).
\]
\end{proposition}

Next is a stronger form of convexity for positive operator monotone functions, as follows.

\begin{proposition}[\cite{MR4156435}*{Theorem 3.1}\footnote{We note an unfortunate typo in that paper, $B^{++}$ in \cite{MR4156435}*{Theorem 3.1} should read $B_{\sa}$.}]\label{prop:KY-main-res}
Let $g(x)$ be an operator monotone function, and $f(x)$ be a positive operator monotone decreasing function, both for $x > 0$.
For any strictly positive map $\Phi\colon M_n(\bC) \to M_k(\bC)$, the map
\[
M_n(\bC)^{++} \to M_k(\bC)_{\sa}, \quad A \mapsto g(\Phi(f(A)))
\]
is convex.
\end{proposition}

We will use the consequence of the above for $g(x) = -x^{-1}$, in the following form: $\Phi(f(A))^{-1}$ is concave in $A \in M_n(\bC)^{++}$ when $f$ is positive operator monotone decreasing.

We also use the Jensen inequality and monotonicity for trace functionals, which can be stated as follows.

\begin{proposition}[\cite{MR1979011}*{Theorem 2.4}]\label{prop:jensen-tr-ineq}
Let $f$ be a convex function defined on some interval $J$, and $C_1, \dots, C_k$ be elements of $M_n(\bC)$ such that $\sum_i C_i^* C_i = I_n$.
Then, for any $A_1, \dots, A_k \in M_n(\bC)_{\sa}$ such that $\sigma(A_i) \subset J$, we have
\[
\Tr f \Bigl( \sum_i C_i^* A_i C_i \Bigr) \le \Tr \Bigl( \sum_i C_i^* f(A_i) C_i \Bigr)
\] 
\end{proposition}

We use this in the following form, by taking $J = (0, \infty)$ and $f = -h$: let $h(x)$ be a concave function for $x > 0$. Then $A \mapsto \Tr h(A)$ is concave for $A \in M_n(\bC)^{++}$.
Thus, in fact the main result of~\cite{MR870784} is enough for us.

\begin{proposition}\label{prop:trace-func-monoton}
Let $f(x)$ be a monotone function with domain $J$.
Suppose that we have $A \le B$ and $\sigma(A), \sigma(B) \subset J$ for $A, B \in M_n(\bC)_{\sa}$.
Then we have
\[
\Tr(f(A)) \le \Tr(f(B)).
\]
\end{proposition}

\begin{proof}
This is well known to experts, but here is a sketch of the proof.
From the minimax principle, we see that the ordered eigenvalues of $A$ and $B$ satisfy $\lambda_i(A) \le \lambda_i(B)$ for $i = 1, \dots, n$.
Collecting the inequalities $f(\lambda_i(A)) \le f(\lambda_i(B))$, we obtain the claim.
\end{proof}

\section{Main result}

When $h(x)$ is a real function defined for $x > 0$, put
\[
\tilde h(x) = - h(x^{-1}).
\]

\subsection{Convexity}
\label{sec:main-conv}

When $h(x)$ is an monotone function for $x > 0$ such that $\tilde h$ is concave and $\lim_{x \to 0} h(x) x = 0$, we put
\[
\breve h(t) = \inf_{x > 0} t x - \tilde h(x) = \inf_{x > 0} t x + h(x^{-1}).
\]
Note that $\breve h$ is well defined as the Legendre transform $\check{\tilde h}$.

Furthermore, we will consider the class of functions $h(x)$ for $x > 0$ satisfying
\begin{equation}\label{eq:func-eq-1}
h(x \amn y) \ge h(x) \amn h(y) \ge h(x \hmn y).
\end{equation}
The first inequality is the usual concavity condition.
The second can be interpreted as concavity of $\tilde h$, hence this class is closed under the transform $h \mapsto \tilde h$.
One motivating example comes from operator monotone functions, as follows.

\begin{proposition}
Suppose that $h(x)$ is operator monotone for $x > 0$.
Then it satisfies \eqref{eq:func-eq-1}.
\end{proposition}

\begin{proof}
This observation can be traced back to \citelist{\cite{MR2805638}}, but let us repeat it here for the reader's convenience.
First, operator monotonicity of $h(x)$ for $x > 0$ implies concavity $h(x \amn y) \ge h(x) \amn h(y)$.
Next, as $\tilde h$ is also operator monotone, it is again concave.
As remarked above, this can be expressed as $h(x \hmn y) \le h(x) \amn h(y)$ up to change of variables.
\end{proof}

\begin{remark}
Recall that any operator monotone function $h(x)$ for $x > 0$ can be written as
\[
h(x) = c_0 + c_1 x + \int \frac{x \lambda - 1}{x + \lambda} d \mu(\lambda)
\]
for some $c_1 \ge 0$ and a finite measure $\mu$ on $[0, \infty)$.
If $\mu$ does not have atom on $0$, we have $\lim_{x \to 0} x h(x) = 0$.
\end{remark}

We are now ready to state and prove our main result.

\begin{theorem}\label{thm:mat-convexity}
Suppose that $h(x)$ is a monotone function for $x > 0$ such that $\tilde h$ is concave, $\lim_{x \to 0} h(x) x = 0$, and that $\breve h$ satisfies \eqref{eq:func-eq-1}.
Let $f(x)$ and $g(x)$ be positive operator monotone decreasing functions for $x > 0$, and let $\Phi\colon M_m(\bC) \to M_k(\bC)$ and $\Psi\colon M_n(\bC) \to M_k(\bC)$ be strictly positive maps.
Then, the map
\begin{equation*}
M_m(\bC)^{++} \times M_n(\bC)^{++} \to \R, \quad (A, B) \to \Tr h\bigl(\Phi(f(A))^{1/2} \Psi(g(B)) \Phi(f(A))^{1/2}\bigr)
\end{equation*}
is jointly convex.
\end{theorem}

\begin{proof}
Let us write $A' = \Phi(f(A))$ and $B' = \Psi(g(B))$.
We have
\[
\Tr h\bigl(\Phi(f(A))^{1/2} \Psi(g(B)) \Phi(f(A))^{1/2}\bigr) = - \Tr\tilde h\bigl(A'^{-1/2} B'^{-1} A'^{-1/2}\bigr).
\]
Thus, it is enough to prove the joint concavity of $\Tr \tilde h(A'^{-1/2} B'^{-1} A'^{-1/2})$ in $A$ and $B$.

We closely follow the proof of \cite{MR3464069}*{Theorem 5.2}.
We first get
\[
\Tr \tilde h\bigl(A'^{-1/2} B'^{-1} A'^{-1/2}\bigr) = \inf_{Y \in M_k(\bC)^{++}} \Tr\bigl(Y A'^{-1/2} B'^{-1} A'^{-1/2} - \breve h(Y)\bigr)
\]
from Proposition \ref{prop:Hia-Tr-h-A-by-Leg-tra}.
Putting $Z = A'^{-1/2} Y A'^{-1/2}$, we can rewrite this as
\begin{equation}\label{eq:rewrite-Tr-tilde-h-brah}
\inf_{Z \in M_k(\bC)^{++}} \Tr\bigl( Z^{1/2} B'^{-1} Z^{1/2} - \breve h(Z^{1/2} A' Z^{1/2}) \bigr).
\end{equation}

Given $A_i \in M_m(\bC)^{++}$ and $B_i \in M_n(\bC)^{++}$ for $i = 1, 2$, let us fix $Z_0 \in M_k(\bC)^{++}$ that almost achieves the infimum \eqref{eq:rewrite-Tr-tilde-h-brah} for $A = A_1 \amn A_2$ and $B = B_1 \amn B_2$.
By Proposition \ref{prop:KY-main-res} applied to the operator monotone function $- x^{-1}$, the map
\[
B \mapsto Z_0^{1/2} B'^{-1} Z_0^{1/2} = \bigl(Z_0^{-1/2} \Psi(g(B)) Z_0^{-1/2}\bigr)^{-1}
\]
is concave, hence we obtain
\[
Z_0^{1/2} B'^{-1} Z_0^{1/2} \ge \bigl(Z_0^{1/2} B_1'^{-1} Z_0^{1/2}\bigr) \amn \bigl(Z_0^{1/2} B_2'^{-1} Z_0^{1/2}\bigr).
\]

As for the term involving $A'$, by assumption on $h$ the function $h' = \tilde{\breve h}$ is concave and monotone.
Thus $C \mapsto \Tr h'(C)$ is concave and monotone for $C \in M_k(\bC)^{++}$ by Propositions \ref{prop:jensen-tr-ineq} and \ref{prop:trace-func-monoton}.
This observation and the concavity of $A \mapsto Z_0^{-1/2} \Phi(f(A))^{-1} Z_0^{-1/2}$ imply that
\[
A \mapsto - \Tr \breve h\bigl(Z_0^{1/2} A' Z_0^{1/2}\bigr) = - \Tr \breve h\bigl(Z_0^{1/2} \Phi(f(A)) Z_0^{1/2}\bigr) 
\]
is concave, hence we obtain
\[
- \Tr \breve h\bigl(Z_0^{1/2} A' Z_0^{1/2}\bigr) \ge \bigl(- \Tr \breve h\bigl(Z_0^{1/2} A_1' Z_0^{1/2}\bigr)\bigr) \amn \bigl(- \Tr \breve h\bigl(Z_0^{1/2} A_2' Z_0^{1/2}\bigr)\bigr).
\]

Thus we see that \eqref{eq:rewrite-Tr-tilde-h-brah} is bounded from below by
\[
\frac12 \mathopen{}\left( \inf_{Z_1, Z_2} \Tr\Bigl( Z_1^{1/2} B_1'^{-1} Z_1^{1/2} - \breve h\bigl(Z_1^{1/2} A_1' Z_1^{1/2}\bigr) \Bigr) + \Tr\Bigl( Z_2^{1/2} B_2'^{-1} Z_2^{1/2} - \breve h\bigl(Z_2^{1/2} A_2' Z_2^{1/2}\bigr) \Bigr) \right),
\]
where $Z_1$ and $Z_2$ separately run over $M_k(\bC)^{++}$.
We thus obtained
\[
\Tr \tilde h\bigl(A'^{-1/2} B'^{-1} A'^{-1/2}\bigr) \ge \Tr \tilde h\bigl(A_1'^{-1/2} B_1'^{-1} A_1'^{-1/2}\bigr) \amn \Tr \tilde h\bigl(A_2'^{-1/2} B_2'^{-1} A_2'^{-1/2}\bigr),
\]
which is what we wanted.
\end{proof}

The above theorem applies for the following cases.
\begin{itemize}
\item $h(x) = \log x$; $\breve h(x) = 1 + \log x$.
\item $h(x) = x^r$ for $0 < r$; $\breve h(x) = r^{1/(r+1)} (1 + r^{-1}) x^{r/(r+1)}$. For $r = 1$ we recover \cite{MR4156435}*{Theorem 4.2}.
\item $h(x) = -x^{-r}$ for $0 < r \le \frac12$; $\breve h(x) = r^{r/(1-r)} (r - 1) x^{r/(r-1)}$.
Put another way,
\[
\Tr\bigl(\Phi(f(A))^{1/2} \Psi(g(B)) \Phi(f(A))^{1/2}\bigr)^{-r}
\]
is concave in $A$ and $B$ for such $r$.
\end{itemize}

\begin{remark}
If $\breve h$ is operator monotone, we can avoid using Propositions \ref{prop:jensen-tr-ineq} and \ref{prop:trace-func-monoton} as the map
\[
M_m(\bC)^{++} \to M_k(\bC)^{++}, \quad A \mapsto -\breve h(Z_0^{1/2} \Phi(f(A)) Z_0^{1/2})
\]
would be operator concave.
The above examples all satisfy this additional assumption.
\end{remark}

\begin{remark}\label{rem:Hiai-conv-II}
By \cite{MR3464069}*{Theorem 5.2},
\[
\Tr h\bigl(\Phi(A^{-p})^{-1/2} \Psi(B^{-q}) \Phi(A^{-p})^{-1/2}\bigr)
\]
is convex if $h(x)$ is a nondecreasing function for $x > 0$ such that either of $h(x^{-(1+p)})$ or $h(x^{-(1+q)})$ is convex.
The above examples of $h$ fall under this setting.
For $h(x) = -x^{-r}$, the bound on $r$ is sharp as seen from the case of $p = q = 1$, $A = B$, and $\Phi(A) = A = \Psi(A)$, see also \cite{MR3067823} for a more precise condition on $h$ that depends on $p$ and $q$.
\end{remark}

\subsection{Concavity}

The concave analogue, which is easier, goes as follows.

\begin{theorem}\label{thm:mat-concavity}
Let $h(x)$ be a concave monotone function for $x > 0$ such that $\lim_{x \to \infty} h(x) x^{-1} = 0$, and that $\check h$ satisfies \eqref{eq:func-eq-1}.
Let $f(x)$ and $g(x)$ be positive operator monotone functions for $x > 0$, and let $\Phi\colon M_m(\bC) \to M_k(\bC)$ and $\Psi\colon M_n(\bC) \to M_k(\bC)$ be strictly positive maps.
Then, the map
\begin{equation*}
M_m(\bC)^{++} \times M_n(\bC)^{++} \to \R, \quad (A, B) \to \Tr h\bigl(\Phi(f(A))^{1/2} \Psi(g(B)) \Phi(f(A))^{1/2}\bigr)
\end{equation*}
is jointly concave.
\end{theorem}

We omit the proof as it is completely analogous to that of Theorem \ref{thm:mat-convexity}.
The above theorem applies for the following cases.
\begin{itemize}
\item $h(x) = \log x$; $\check h(x) = 1 + \log x$.
\item $h(x) = x^r$ for $0 < r \le \frac12$; $\check h(x) = r^{r/(1-r)} (r - 1) x^{r/(r-1)}$.
\item $h(x) = -x^{-r}$ for $0 < r$; $\check h(x) = r^{1/(r+1)} (1 + r^{-1}) x^{r/(r+1)}$.
\end{itemize}

\section{\texorpdfstring{C$^*$}{C*}-algebraic setting}
\label{sec:C-star}

The above results have straightforward generalization to the setting of unital C$^*$-algebras with tracial states.
In this section $\fA$, $\fB$, and $\fC$ denote unital C$^*$-algebras, and $\tau$ denotes a tracial state on $\fC$.
We use notations such as $\fA^{++}$ and $\fA_{\sa}$ analogous to the case of matrix algebras.

First let us establish a generalization of Proposition \ref{prop:Hia-Tr-h-A-by-Leg-tra} to this setting.

\begin{proposition}
Let $h(x)$ be a concave function for $x > 0$ such that $\lim_{x \to \infty} h(x) x^{-1} = 0$.
For any $A \in \fC^{++}$, we have
\[
\tau(h(A)) = \inf_{B \in \fC^{++}} \tau(A B - \check h(B)). 
\]
\end{proposition}

\begin{proof}
Let $\fM$ be the von Neumann algebraic closure of $\fC$ in the GNS representation associated with $\tau$.
We denote the extension of $\tau$ to $\fM$ again by $\tau$.
Let $\fN$ be the von Neumann subalgebra of $\fM$ generated by the image of $A$.
Then there is a (unique) $\tau$-preserving conditional expectation $E \colon \fM \to \fN$.

As $-\check h(x)$ is convex for $x > 0$, and $E$ is unital positive map, we have
\[
-\tau(\check h(E(B))) \le -\tau(E(\check h(B))) = -\tau(\check h(B))
\]
for any $B \in \fM^{++}$ by \cite{MR870784}*{Corollary}.
Combined with $\tau(A B) = \tau(A E(B))$, we obtain
\[
\inf_{B \in \fM^{++}} \tau(A B - \check h(B)) = \inf_{B \in \fN^{++}} \tau(A B - \check h(B)).
\]
The right hand side is equal to $\inf_{f} \tau(A f(A) - \check h(f(A)))$, where $f$ runs over bounded nonnegative Borel measurable functions on $\sigma(A)$, hence it is equal to $h(A)$.
By a standard approximation argument, the same infimum is achieved when $f$ runs over nonnegative continuous functions on $\sigma(A)$.

Finally, together with the obvious inequality
\[
\inf_{B \in \fM^{++}} \tau(A B - \check h(B)) \le \inf_{B \in \fC^{++}} \tau(A B - \check h(B)),
\]
we obtain the claim.
\end{proof}

Proposition \ref{prop:KY-main-res} holds in this setting, as \cite{MR4156435}*{Theorem 3.1} was already proved for in such generality.
The rest is quite well known, as follows.

\begin{proposition}[\cite{MR2606886}]
Let $f$ be a convex function defined on some interval $J$, and $C_1, \dots, C_k$ be elements of $\fA$ such that $\sum_i C_i^* C_i = 1$.
Then, for any $A_1, \dots, A_k \in \fA_{\sa}$ such that $\sigma(A_i) \subset J$, we have
\[
\tau \Bigl( f \Bigl( \sum_i C_i^* A_i C_i \Bigr) \Bigr) \le \tau \Bigl( \sum_i C_i^* f(A_i) C_i \Bigr)
\] 
\end{proposition}

Again the setting of \cite{MR870784} is enough for us, as we only need to deal with convex functions defined for $x > 0$.

\begin{proposition}[\cite{MR796042}*{Theorem 2}]
Let $f(x)$ be a monotone function with domain $J$.
Suppose that we have $A \le B$ and $\sigma(A), \sigma(B) \subset J$ for $A, B \in \fA_{\sa}$.
Then we have
\[
\tau(f(A)) \le \tau(f(B)).
\]
\end{proposition}

With above results at hand, the proof of our main results carry over to C$^*$-algebraic setting, and we obtain the following.

\begin{theorem}
Let $f(x)$, $g(x)$, and $h(x)$ be real functions for $x > 0$ as in Theorem \ref{thm:mat-convexity}.
Let $\Phi\colon \fA \to \fC$ and $\Psi\colon \fB \to \fC$ be strictly positive maps.
Then the map
\begin{equation}
\fA^{++} \times \fB^{++} \to \R, \quad (A, B) \to \tau \bigl( h\bigl(\Phi(f(A))^{1/2} \Psi(g(B)) \Phi(f(A))^{1/2}\bigr) \bigr)
\end{equation}
is jointly convex.
\end{theorem}

\begin{theorem}
Let $f(x)$, $g(x)$, and $h(x)$ be real functions for $x > 0$ as in Theorem \ref{thm:mat-concavity}.
Let $\Phi\colon \fA \to \fC$ and $\Psi\colon \fB \to \fC$ be strictly positive maps.
Then, the map
\begin{equation*}
\fA^{++} \times \fB^{++} \to \R, \quad (A, B) \to \tau \bigl( h\bigl(\Phi(f(A))^{1/2} \Psi(g(B)) \Phi(f(A))^{1/2}\bigr) \bigr)
\end{equation*}
is jointly concave.
\end{theorem}

\end{document}